\documentclass [a4paper, 11pt, reqno]{amsart}
\usepackage{amssymb}

\usepackage{geometry}
\newcommand{\R} {\ensuremath{\mathbb{R}}}

\newcommand{\C} {\ensuremath{\mathbb{C}}}
\newcommand{\Z} {\ensuremath{\mathbb{Z}}}

\renewcommand{\o}[1]{\overline{#1}}

\newcommand{\dq}{\overline{\partial}}
\DeclareMathOperator{\Reg}{Reg}
\DeclareMathOperator{\Sing}{Sing}
\DeclareMathOperator{\Dom}{Dom}
\DeclareMathOperator{\codim}{codim}

\newtheorem {satz} {Satz} [section]
\newtheorem {lem} [satz] {Lemma}
\newtheorem {cor} [satz] {Corollary}

\newtheorem {thm} [satz] {Theorem}

\numberwithin{equation}{section}

\DeclareMathOperator{\supp}{supp}

\title[Koppelman formulas on affine cones] 
{Koppelman formulas on affine cones over smooth projective complete intersections}

\author{R. L\"ark\"ang}
\address{Richard L\"ark\"ang, Department of Mathematics, University of Wuppertal, Gau{\ss}str. 20, 42119 Wuppertal, Germany, and Department of
  Mathematics, Chalmers University of Technology and the University of Gothenburg, 412 96 G\"oteborg, Sweden.}
\email{larkang@chalmers.se}

\author{J. Ruppenthal}
\address{Jean Ruppenthal, Department of Mathematics, University of Wuppertal, Gau{\ss}str. 20, 42119 Wuppertal, Germany.}
\email{ruppenthal@uni-wuppertal.de}

\date{\today}

\subjclass[2000]{32A26, 32A27, 32B15, 32C30, 32W05}

\begin{document}

\begin{abstract} 
In the present paper, we study regularity of the Andersson--Samuelsson Koppelman integral operator on affine cones
over smooth projective complete intersections.
Particularly, we prove $L^p$- and $C^\alpha$-estimates, and compactness of the operator, when the degree is sufficiently small.
As applications, we obtain homotopy formulas for different $\dq$-operators acting on $L^p$-spaces of forms, including the case $p=2$
if the varieties have canonical singularities. We also prove that the $\mathcal{A}$-forms introduced by Andersson--Samuelsson
are $C^\alpha$ for $\alpha < 1$.
\end{abstract}

\maketitle

~\\[-16mm]
\section{Introduction}

In $\C^n$, it is classical that the $\dq$-equation $\dq f = g$, where $g$ is a $\dq$-closed $(0,q)$-form,
can be solved locally for example if $g$ is in $C^\infty$, $L^p$ or $g$ is a current, where the solution
$f$ is of the same class (or in certain cases, also with improved regularity).
To prove the existence of solutions which are smooth forms or currents, or to obtain
$L^p$-estimates for smooth solutions, one can use Koppelman formulas, see for example, \cite{Ra},\cite{LiMi}.

On singular varieties, it is no longer necessarily the case that the $\dq$-equation is locally
solvable over these classes of forms, as for example on the variety $\{ z_1^4 + z_2^5 + z_2^4 z_1 = 0 \}$,
there exist smooth $\dq$-closed forms which do not have smooth $\dq$-potentials,
see e.g. \cite[Beispiel~1.3.4]{RuDipl}.

Solvability of the $\dq$-equation on singular varieties has been
studied in various articles in recent years, for example describing
in certain senses explicitly the obstructions to solving the $\dq$-equation
in $L^2$, see \cite{FOV},\cite{OV2},\cite{RDuke}.
Among these and other results, one can find examples when the $\dq$-equation
is not always locally solvable in $L^p$, for example when $p = 1$ or $p = 2$.

On the other hand, in \cite{AS}, Andersson and Samuelsson define on an arbitrary
pure dimensional singular variety $X$ sheaves $\mathcal{A}^X_q$ of $(0,q)$-currents,
such that the $\dq$-equation is locally solvable in $\mathcal{A}^X$, and the solution is
given by Koppelman formulas, i.e., there exists operators $\mathcal{K} : \mathcal{A}_{q}^X \to \mathcal{A}_{q-1}^X$
and $\mathcal{P} : \mathcal{A}^X_0 \to \mathcal{O}_X$, such that
if $\varphi \in \mathcal{A}_q^X$, then
\begin{equation} \label{eq:koppel}
    \varphi = \dq \mathcal{K}\varphi + \mathcal{K} (\dq \varphi),
\end{equation}
locally in the sense of distributions if $q \geq 1$, and 
\begin{equation} \label{eq:koppel2}
    \varphi = \mathcal{P} \varphi + \mathcal{K} (\dq \varphi),
\end{equation}
locally in the sense of distributions if $q = 0$,
where the operators $\mathcal{K}$ and $\mathcal{P}$ are given as principal value integral operators
\begin{equation} \label{eq:AS-def}
    \mathcal{K}\varphi(z) = \int K(\zeta,z) \wedge \varphi(\zeta) \text{ and }
    \mathcal{P}\varphi(z) = \int P(\zeta,z) \wedge \varphi(\zeta),
\end{equation}
for some integral kernels $K(\zeta,z)$ and $P(\zeta,z)$.
On $X^* = \Reg X$, the regular part of $X$, the sheaf $\mathcal{A}^X_q$ coincides with the sheaf of smooth $(0,q)$-forms.
For the cases when the $\dq$-equation is not solvable for smooth forms, 
the $\mathcal{A}$-sheaves must necessarily have singularities along $\Sing X$,
but from the definition of the $\mathcal{A}$-sheaves, it is not very apparent
how the singularities of the $\mathcal{A}$-sheaves are in general.
In order to take better advantage of the results in \cite{AS}, one would
like to know more precisely how the singularities of the
$\mathcal{A}$-sheaves look like.
In particular, it would be interesting to know whether for certain varieties,
the $\mathcal{A}$-sheaves are in fact smooth, or, say, $C^k$ also over $\Sing X$.

Our motivation for studying the $\dq$-equation using Koppelman formulas is two-fold:
First of all, as in the smooth case, using integral formulas for studying the $\dq$-equation
has the advantage that it can be used for understanding the $\dq$-equation over various function spaces,
like forms which are $C^k$, $C^\infty$, H\"older, $L^p$ or currents. 
As mentioned above, a large part of the study of the $\dq$-equation on singular
varieties has been restricted to $L^2$-spaces, while using integral formulas,
we can indeed obtain new results about solvability also in $L^p$-spaces for $p \neq 2$.
In addition, it is often easy to prove that integral operators are compact,
and indeed, we do indeed here obtain compact solution operators for the $\dq$-equation.

A second motivation is the following: the $\mathcal{A}$-sheaves in \cite{AS}
are defined by starting with smooth forms, applying Koppelman operators, multiplying
with smooth forms, applying Koppelman operators, and iterating this procedure a
finite number of times. We obtain here that for the varieties we study,
the $\mathcal{A}$-sheaves are contained in the sheaves of forms with $C^\alpha$ coefficients,
for any $\alpha < 1$, see Corollary~\ref{cor:asheaves} below.

In this article, we consider Koppelman type integral formulas for the $\dq$-equation on affine cones
over smooth projective complete intersections of low enough degree. More precisely,
let $X = \{ \zeta \in \C^N \mid h(\zeta) = 0 \}$ be a subvariety of dimension $n = N-\nu$,
where $h = (h_1,\dots,h_\nu)$ is a tuple of homogeneous polynomials of degrees $(d_1,\dots,d_\nu)$.
We let $d := d_1 + \dots + d_\nu$ be the degree of $X$, and assume that $d\leq 2n+\nu-1$ and that 
$X$ has an isolated singularity at the origin $\{ 0 \}$. Equivalently, if $Y \subseteq \mathbb{P}^{N-1}$
is a smooth projective complete intersection of degree $d$ defined by $Y := \{ [z] \in \mathbb{P}^{N-1} \mid h(z) = 0 \}$,
then, $X$ is the affine cone over $Y$.
In \cite{LR}, we studied similar problems for the special case of the so-called $A_1$-singularity,
which is the subvariety $X = \{ \zeta \in \C^3 \mid \zeta_1^2 + \zeta_2^2 + \zeta_3^2 = 0 \}$.

For general varieties, the operators \eqref{eq:AS-def} from \cite{AS} only exist as principal value
operators, and hence require some smoothness of the input, but our first main result is that for
the varieties we consider in this article, we can extend the operators to work on $L^p$-forms.
For precise definitions of what we mean by $L^p$-forms, $C^\alpha$-forms and $C^{0,1}$-functions on $D'$ and $D$,
see Section \ref{sec:lp-forms}.

\begin{thm}\label{thm:main1}
    Assume that $X \subseteq \C^N$ is the affine cone over a smooth projective complete intersection $Y \subseteq \mathbb{P}^{N-1}$
    of degree $d \leq 2n+\nu-1$, where $n = \dim X$ and $\nu = \codim X = N-n$.
    Let $\Omega \subset\subset \Omega' \subset \subset \C^N$
    be two strictly pseudoconvex domains, and let $D := X \cap \Omega$ and $D' := X \cap \Omega'$.
    Let $\mathcal{K}$ and $\mathcal{P}$ be the integral operators from \cite{AS} on $D'$, as here defined in
    \eqref{eq:AS1} and \eqref{eq:AS2}, and assume that 
    $$\frac{2n}{2n-(d-\nu)} < p \leq \infty$$ and $q \in \{1,\dots,n\}$.
    Then:
    
    \medskip
    (i) $\mathcal{K}$ gives a bounded compact linear operator from $L^p_{0,q}(D')$ to $L^p_{0,q-1}(D)$.
    
    \medskip
    (ii) $\mathcal{K}$ gives a continuous compact linear operator from $L^\infty_{0,q}(D')$ to $C^\alpha_{0,q-1}(\overline{D})$
    for $0 \leq \alpha < 1$.

    \medskip
    (iii) $\mathcal{P}$ gives a continuous compact linear operator from $L^1_{0,0}(D')$ to $C^{0,1}(\overline{D})$.
\end{thm}

In particular, one obtains the following result about the $\mathcal{A}$-sheaves from \cite{AS}.

\begin{cor} \label{cor:asheaves}
    Let $X$ and $D$ be as in Theorem~\ref{thm:main1}, and let, as in \cite{AS}, $\mathcal{A}^X_{q}$ be the sheaf of currents
    which can be locally written as a finite sum of currents of the form 
    \begin{equation*}
        \xi_{\nu+1}\wedge (\mathcal{K}_\nu(\dots \xi_3 \wedge \mathcal{K}_2(\xi_2\wedge \mathcal{K}_1(\xi_1)))),
    \end{equation*}
    where each $\mathcal{K}_i$ is an integral operator as in Theorem \ref{thm:main1},
    mapping forms on $D_i' := \Omega_i \cap X$ to forms on $D_{i+1}'$,
    where $\Omega = \Omega_{\nu+1} \subset\subset \Omega_\nu \subset\subset \dots \subset\subset \Omega_1 \subset\subset \C^N$
    are strictly pseudoconvex domains, and $\xi_i$ are smooth forms on $D_i'$. Then
    \begin{equation*}
        \mathcal{A}^X_q(D) \subseteq C^\alpha_{0,q}(D)
    \end{equation*}
    for any $0 \leq \alpha < 1$.
\end{cor}

Although by Theorem~\ref{thm:main1} the Koppelman operator $\mathcal{K}$ maps $L^p_{0,q}(D')$ to $L^p_{0,q-1}(D)$ for $p > 2n/\big(2n-(d-\nu)\big)$,
this does not necessarily imply that the $\dq$-equation is locally solvable in $L^p$ for such $p$,
since it is not necessarily the case that \eqref{eq:koppel} holds on $D$ for $\varphi \in L^p(D')$.
However, in order to describe when the Koppelman formula \eqref{eq:koppel} does indeed hold,
we first need to discuss various definitions of the $\dq$-operator on $L^p$-forms
on singular varieties. We let $D \subseteq X$ be some open set, and we let $\dq_{sm}$ be the $\dq$-operator
on smooth $(0,q)$-forms with support on $D^*=D\setminus\{0\}$ away from the singularity.
This operator has various extensions as a closed operator in $L^p_{0,q}(D)$.

One extension of the $\dq_{sm}$-operator is the maximal closed extension, i.e., the weak $\dq$-operator $\dq_w^{(p)}$
in the sense of currents, so if $g \in L^p_{0,q}(D)$, then $g \in \Dom \dq_w^{(p)}$ if $\dq g \in L^p_{0,q+1}(D)$
in the sense of distributions on $D$.
\footnote{This is what we take as definition of $\dq_w^{(p)}$ on $D$. However, to be precise, this definition only 
coincides with the maximal closed extension of $\dq_{sm}$ for $p \geq 2n/(2n-1)$, which is the only case of interest to us.
In general, that $\varphi$ lies in the domain of the maximal closed extension of $\dq_{sm}$ means that $\dq \varphi|_{D^*} \in L^p(D^*)$.
When $p \geq 2n/(2n-1)$, it then follows that $\dq \varphi \in L^p(D)$, see \cite[Satz~4.3.3]{RuThesis}.}
When it is clear from the context, we will drop the superscript $(p)$ in $\dq^{(p)}_w$,
and we will for example write $g \in \Dom \dq_w \subset L^p_{0,q}(D)$. For the $\dq_w$-operator, we obtain the following
result about the Koppelman formulas \eqref{eq:koppel} and \eqref{eq:koppel2}.

\begin{thm}\label{thm:main3}
    Let $X$, $D'$, $D$, $\mathcal{K}$ and $\mathcal{P}$ be as in Theorem~\ref{thm:main1}.
    Let $\varphi \in \Dom \dq_w \subseteq L^p_{0,q}(D')$, where 
    $$\frac{2n}{2n-(d-\nu+1)} \leq p \leq \infty$$
    and $q \in \{0,\dots,n\}$.
    Then
    \begin{eqnarray} \label{eq:dbarlp}
        \varphi &=& \left\{ \begin{array}{ll}
            \dq_w \mathcal{K}\varphi + \mathcal{K}\big( \dq_w \varphi\big) & \text{ if $q \geq 1$, } \\
            \mathcal{P} \varphi + \mathcal{K}\big( \dq_w \varphi\big) & \text{ if $q = 0$, }
        \end{array}\right.
    \end{eqnarray}
    in the sense of distributions on $D$.
\end{thm}

Note in particular, if $d \leq N-1 = n+\nu-1$, then \eqref{eq:dbarlp} holds in the important case $p=2$.
By \cite[Corollary~3.3]{Kol}, the condition $d \leq N-1$ means precisely that $X$ has canonical singularities,
which is an important class of singularities in the minimal model program. As we explain below, this result is
indeed optimal with respect to the condition on $d$ in the case $p=2$, since the $\dq_w$-equation is not
solvable for $(0,n-1)$-forms if $d \geq N$.

Another extension of the $\dq$-operator is the minimal closed extension, i.e., the
strong extension $\dq_s^{(p)}$ of $\dq_{sm}$, which is the graph
closure of $\dq_{sm}$ in $L^p_{0,q}(D) \times L^p_{0,q+1}(D)$, so
$\varphi \in \Dom \dq_s^{(p)} \subset L^p_{0,q}(D)$, if there exists a sequence of smooth 
forms $\{\varphi_j\}_j \subset L^p_{0,q}(D)$ with support away from the singularity, i.e.,
$$\supp \varphi_j \cap \{0\} = \emptyset,$$
such that
\begin{eqnarray*}
\varphi_j \rightarrow \varphi \ \ \ &\mbox{ in }& \ \ L^p_{0,q}(D),\\
\dq \varphi_j \rightarrow \dq \varphi \ \ \ &\mbox{ in }& \ \ L^p_{0,q+1}(D)\label{eq:dbars2}
\end{eqnarray*}
as $j\rightarrow \infty$.

For the strong $\dq$-operator, we obtain the following.

\begin{thm}\label{thm:main4}
    Let $X$, $D'$, $D$ and $\mathcal{K}$ be as in Theorem~\ref{thm:main1}, and assume that $X$
    has degree $d < 2n+\nu-1$, and that $D$ has smooth boundary.
    Let $\varphi\in \Dom \dq_s \subseteq L^p_{0,q}(D')$, $1\leq q \leq n$, where 
    $$\frac{2n}{2n-(d-\nu)} < p \leq 2n.$$
    Then
    \begin{eqnarray*}
        \mathcal{K} \varphi &\in& \Dom\dq_s \subset L^p_{0,q-1}(D).
    \end{eqnarray*}
\end{thm}

As a corollary, we thus obtain that the Koppelman formula holds also for the $\dq_s$-operator.

\begin{cor}\label{cor:main4}
    Let $X$, $D'$, $D$ and $\mathcal{K}$ be as in Theorem~\ref{thm:main1}, and assume that $X$
    has degree $d < 2n+\nu-1$.
    Let $\varphi \in L^p_{0,q}(D')$ such that $\varphi \in \Dom \dq_s$, where $q \in \{1,\dots,n\}$ and $$\frac{2n}{2n-(d-\nu)} <p \leq 2n.$$
    Then
    \begin{eqnarray*}
        \varphi = \dq_s \mathcal{K}\varphi + \mathcal{K}\big( \dq_s \varphi\big)
    \end{eqnarray*}
    in the sense of distributions on $D$.
\end{cor}

For $p = 2$, this result is optimal with respect to $d$ in the same sense as for $\dq_w$ in Theorem~\ref{thm:main3}.

The setting in \cite{AS} is rather different compared to this article, since here, we are mainly concerned
with forms on $X$ with coefficients in $L^p$, while in \cite{AS}, the type of forms
considered, denoted $\mathcal{W}^X_q$, are generically smooth, and have in a certain sense ``holomorphic singularities''
(like for example the principal value current $1/f$ of a holomorphic function $f$),
but there is no direct growth condition on the singularities.
For the precise definition of the class $\mathcal{W}^X_q$, we refer to \cite{AS}. In the setting of \cite{AS},
the $\dq$-operator $\dq_X$ considered there is different from the ones considered
here, $\dq_s$ and $\dq_w$. For currents in $\mathcal{W}^X_q$, one can define the product with certain ``structure forms''
$\omega_X$ associated to the variety. A current $\mu \in \mathcal{W}^X_q$ lies in $\Dom \dq_X$ if there exists a current
$\tau \in \mathcal{W}^X_{q+1}$ such that $\dq (\mu \wedge \omega) = \tau \wedge \omega$ for all structure forms $\omega$.
(To be precise, this formulation works when $X$ is Cohen-Macaulay, as is the case for example here, when $X$ is a 
complete intersection).

Combining our results about $\mathcal{K}$ and the $\dq_w$- and $\dq_s$-operator
with some properties about the $\mathcal{W}^X$-sheaves,
we obtain results similar to Theorem~\ref{thm:main4} for the $\dq_X$-operator,
answering in part a question in \cite{AS} (see the paragraph at the end of page 288 in \cite{AS}).

\begin{thm}\label{thm:main5}
Let $X$, $D'$, $D$ and $\mathcal{K}$ be as in Theorem~\ref{thm:main1}, and assume that $X$
has degree $d < 2n+\nu-1$.
Let $\varphi\in \Dom \dq^{(p)}_s \cap \mathcal{W}^X_q(D')$, $1\leq q \leq n$, where
$2n/(2n-(d-\nu)) <p \leq 2n.$
Then
\begin{eqnarray*}
\mathcal{K} \varphi &\in& \Dom\dq_X.
\end{eqnarray*}
\end{thm}

When $X$ is as in Theorem~\ref{thm:main5}, then the structure form on $X$ will locally behave
like $1/\|\zeta\|^{d-\nu}$ in $\C^n$, see \eqref{eq:str-form-estimate}.
Thus, $\omega \in L^{p^*}_{n,0}(D)$ for all $1 \leq p^* < 2n/(d-\nu)$.
The conclusions of Theorem~\ref{thm:main5} mean that
\begin{equation*}
    \dq (\mathcal{K} \varphi \wedge \omega_X) = (\dq \mathcal{K} \varphi) \wedge \omega_X.
\end{equation*}
Since $\varphi \in \Dom \dq_s \subseteq L^p(D')$, by the Koppelman formula for $\dq_w$ on $L^p$,
we get that $\dq \mathcal{K} \varphi \in L^p(D)$. As $p > 2n/(2n-(d-\nu))$, we have $p^* := p/(p-1) < 2n/(d-\nu)$,
and so, by the discussion above, $\omega \in L^{p^*}_{n,0}(D)$.
Thus, the products
$\mathcal{K}\varphi \wedge \omega_X$ and $(\dq \mathcal{K} \varphi) \wedge \omega_X$ exist
(almost-everywhere) pointwise and lie in $L^1_{n,*}(D)$ by H\"older's inequality.

The proof of Theorem~\ref{thm:main5} is essentially the same as the proof of Theorem~1.6 in \cite{LR}.
The only differences are that here, as described above, one uses Corollary~\ref{cor:main4} to conclude that $\dq \mathcal{K} \varphi \in L^p$,
and at the point where H\"older's inequality is used, one uses that if $p^* := p/(p-1)$, then as explained
above, $\omega_X \in L^{p^*}(D)$.

When $X = \{ \zeta_1^2 + \zeta_2^2 + \zeta_3^2 = 0 \} \subseteq \C^3$ is the so-called $A_1$-singularity,
we proved in \cite{LR} that if $\varphi \in \Dom \dq_w^{(2)}$, then $\mathcal{K}\varphi \in \Dom \dq_s^{(2)}$,
and as a consequence of this result and the Koppelman formula for $\dq_w^{(2)}$, we then obtained that
$\dq_w^{(2)}$ and $\dq_s^{(2)}$ coincide on the $A_1$-singularity.
In Theorem~\ref{thm:main4}, we require the stronger assumption that $\varphi$ is in $\Dom \dq_s^{(p)}$,
and we can then not conclude that $\dq_w^{(p)}$ and $\dq_s^{(p)}$ coincide on the varieties that we consider.
Theorem~\ref{thm:main4} is however strong enough to obtain the Koppelman formula for $\dq_s$.

\bigskip 
The following results about solvability of the $\dq$-equation $\dq f = g$, when $\dq g = 0$,
on affine homogeneous varieties with an isolated singularity can be found in earlier works.
By the phrase that "there exists $f$" in a certain function space for $g$ with certain properties,
we shall always mean that $\dq f=g$.
Throughout this discussion, we let as above, $X \subseteq \C^N$ be an analytic variety of pure dimension $n$,
and let $D \subset \subset D' \subset\subset X$ be two domains, which are intersections of $X$ with strictly pseudoconvex 
domains in $\C^N$ (in some cases $D$ and $D'$ should be intersections of $X$ with balls in $\C^N$).
Recall also, as mentioned above:
when $X$ is the affine cone of a smooth projective complete intersection in $\mathbb{P}^{N-1}$ of degree $d$,
then $X$ has a canonical singularity at $0$ if and only if $d \leq N-1$. 

First of all, Henkin and Polyakov \cite{HP} showed that for any complete intersection, if $g \in C^\infty_{0,q}(D)$,
then there exists $f \in C^\infty_{0,q-1}(D^*)$, where $D^* = D \setminus \Sing X$. 

We now consider the $\dq_w$-operator.
If $X$ is an arbitrary variety, which is Cohen-Macaulay (so in particular, if $X$ is a complete intersection),
with an isolated singularity at $0$, then Forn{\ae}ss, {\O}vrelid, Vassiliadou showed that for $g \in L^2_{0,q}(D)$,
where $1 \leq q \leq n-2$, there exist $f \in L^2_{0,q-1}(D)$, and the case $q = n$ is treated in \cite{OR}
(also without the Cohen-Macaulay assumption).

For weighted homogeneous varieties, if $g$ has compact support in $D$, and $g \in L^p_{0,q}(D^*)$,
then for $1 \leq q \leq n$, and $1 \leq p \leq \infty$, by \cite{RZ2}, there exists $f \in L^p_{0,q-1}(D^*)$.
If $X$ is homogeneous with isolated singularities, $g \in L^\infty_{0,1}(D)$, still with compact support,
then $f \in C_{0,0}^\alpha(D)$ for any $\alpha < 1$.
If $X$ is as in Theorem~\ref{thm:main1}, and $d = n$, then for $g \in L^p_{0,q}(D)$, where $1 \leq p \leq \infty$ and
$q \leq n-2$, there exist $f \in L^p_{0,q-1}(D)$ by \cite{RMatZ}, Theorem 6.5.

If we now turn to the $\dq_s$-operator, by \cite{RSerre}, $(L^{2,loc}_{0,q},\dq_s)$ is a resolution of
$\mathcal{O}_{X,x}$ if and only if $x \in X$ has rational singularities.
Thus, if $D \subset\subset D'$, and $D'$ is strictly pseudoconvex, if $g \in \ker \dq_s \subseteq L^2_{0,q}(D')$,
there thus exists $f \in L^2_{0,q-1}(D)$ if $(X,0)$ is a rational singularity.
On the other hand, if $(X,0)$ is not a rational singularity, then there exist a neighborhood $D'$ of
$0$ and $g \in L^2_{0,q-1}(D)$ such that there does not exist any $f \in L^2_{0,q-1}(D)$ for any neighborhood
$D$ of $0$.
When $(X,x)$ is Cohen-Macaulay, then $(X,x)$ has rational singularities if and only if $(X,x)$ has canonical singularities,
see \cite[p. 85]{Kol}.

Finally, one can also compare solvability with respect to the $\dq_s$ and $\dq_w$-operator.
By \cite{RDuke} and \cite{RSerre}, the $L^{2,loc}$-cohomologies on $X$ coincide when one considers either the $\dq_s$-
or the $\dq_w$-operator for $X$ being the affine cone of a smooth projective complete intersection, because the blow-up
of the origin is then a resolution of singularities of $X$, and the exceptional divisor has multiplicity $1$.
Thus, also the $\dq_w$-equation is locally solvable for all $g \in L^2_{0,q}$ and all $1 \leq q \leq n$ if and only
if $d \leq N-1$.

To conclude, we see that when $g$ does not have compact support, our results about solvability
in $L^p_{0,q}$ for $p \neq 2$, appear new when $d \neq n$ or $q \geq n-1$. 

In regards to optimality of our results, for $p = 2$, we see by the discussion above, that the $\dq_w$-
and the $\dq_s$-equation are locally solvable for all $g \in \ker \dq_s \subseteq L^2_{0,q}$ or $g \in \ker \dq_w \subseteq L^2_{0,q}$
when $q \neq n-1$ for any affine cone of a smooth projective complete intersection of arbitrary degree,
and for $q = n-1$ if and only if $d \leq N-1$.
Thus, for $p = 2$, Theorem~\ref{thm:main3} and Corollary~\ref{cor:main4} are optimal in the sense that
they give solutions for all $1 \leq q \leq n$ exactly for those affine cones over a smooth projective
complete intersection for which solutions always exist.

We mention here how our results and methods are related to the ones in \cite{LR}. In \cite{LR},
we obtained results similar to the results here, for the special case of the so-called $A_1$-singularity
$X = \{ \zeta \in \C^3 \mid \zeta_1^2 + \zeta_2^2 + \zeta_3^2 = 0 \}$.
The methods are however a bit different. In \cite{LR}, we used a two-sheeted branched covering $\pi : \C^2 \to X$
of $X$ to essentially reduce the problem to similar problems in the case when $X = \C^2$.
Here now, we instead consider the problem, and estimate integrals directly on the variety $X \subseteq \C^N$,
using some basic estimates regarding radial integrals in Section~\ref{sect:basic-estimates}.
Since we do not make any assumptions on the variety in Section~\ref{sect:basic-estimates}
(except for being of pure dimension), such a method has the hope of working more generally.
In addition, even though we could in \cite{LR} reduce the problem to integral operators
in $\C^2$, the method still became rather involved, as we first of all needed to consider weighted
$L^p$-spaces on $\C^2$, and in addition, the integral kernels that we needed to study became 
rather complicated.

\bigskip
The present paper is organised as follows. We start by providing basic integral estimates
on arbitrary analytic varieties in Section \ref{sect:basic-estimates}, and the definition of
$C^\alpha$- and $L^p$-forms on singular spaces in Section \ref{sec:lp-forms}.
In Section \ref{sec:Iestimates}, we prove the relevant estimates for integral operators with isotropic isolated
poles on varieties with arbitrary singularities, while in Section \ref{sec:cut-off},
we study how $L^p$-forms on a singular variety can be approximated by smooth forms
(which is needed to apply the Andersson--Samuelsson homotopy formula).
Finally, in Section \ref{sec:main},
we recall the Koppelman formulas of Andersson--Samuelsson and prove the main theorems of this paper.

\section{Basic integral estimates on analytic varieties} \label{sect:basic-estimates}

Let $X\subset \C^N$ be an analytic variety of pure dimension $n$. 
We consider $X$ as a Hermitian complex space with the restriction of the standard metric from $\C^N$,
i.e., the regular part $X^*:=\Reg X$ of $X$ carries the induced Hermitian metric. 
With respect to the volume element induced by this metric, the singular part $\Sing X$ is a null set,
and we denote by $dV_X$ the extension to $X$ of the volume element on $X^*$.
Let $B_r(z)$ be the ball of radius $r>0$ centered at the point $z\in\C^N$.

\smallskip
\subsection{Estimates of radial functions on analytic varieties}

Let $f : Y \to \R_{\geq 0}$ be a positive measurable function on a measure space $(Y,\mu)$. We define
the distribution function of $f$ as
\begin{equation*}
    \lambda_f(t) := \mu(\{ y \in Y \mid f(y) \geq t \}).
\end{equation*}
Our use for distribution functions is the following result:
\begin{equation} \label{eq:estimate0}
    \int_Y f(y) d\mu(y) = \int_0^\infty \lambda_f(t) dt,
\end{equation}
provided the integral exists. The proof of \eqref{eq:estimate0} follows directly from
writing $f(y) = \int_0^{f(y)} dt$ in the left-hand side of \eqref{eq:estimate0},
and changing the order of integration.

We will now let $Y$ be the set $X \cap (B_{r_2}(z) \setminus \overline{B_{r_1}(z)})$ for $r_2 \geq r_1 \geq 0$.
We want to estimate integrals of the form
\begin{equation*}
    \int_Y \frac{1}{\|\zeta-z\|^\alpha} dV_X(\zeta),
\end{equation*}
where $\alpha\geq0$.
To do this, we begin by estimating the distribution function of $f(\zeta) = 1/\|\zeta-z\|^\alpha$ on $Y$.
First of all, we have by \cite{De}, Consequence~III.5.8, that
if we write
\begin{equation*}
    \int_{X \cap B_r(z)} dV_X(\zeta) = v(r,z) r^{2n},
\end{equation*}
then $v(r,z)$ is increasing in $r$.
We let $K$ be some compact subset of $X$ and let $R > 0$ be fixed.
Then there exists some $C$ such that
$v(r,z) \leq C$ for any $z \in K$ and $r < R$.

In addition, by \cite{De}, Theorem III.7.7,
there exists some constant $c$ such that $0< c \leq \lim_{r\to 0+} v(r,z)$ independently of $z \in X$.
Thus, for $z \in K$ and $0 \leq r \leq R$, we get that
there exists constants $c,C$ such that
\begin{equation}\label{eq:De2}
    0 < c \leq v(r,z) \leq C.
\end{equation}

Using \eqref{eq:De2}, we can estimate integrals of radial functions on a variety $X$ of dimension $n$
in terms the corresponding integral on $\C^n$.

\begin{lem} \label{lem:comparable}
    Let $X \subseteq \C^N$ be an analytic subvariety of pure dimension $n$.
    Let $K \subseteq X$ be compact, and let $z \in K \subseteq X$ and $R > 0$ be fixed.
    Assume that $f : X \to \R_{\geq 0}$ is of the form $f(\zeta) = g(|\zeta-z|)$
    for some function $g : \R_{\geq 0} \to \R_{\geq 0}$.
    Let $\tilde{f} : \C^n \to \R_{\geq 0}$
    be defined by $\tilde{f}(\zeta) = g(|\zeta|)$. Then, for $r \leq R$,
    \begin{equation*}
        c \int_{B_r(0)} \tilde{f}(\zeta) dV_{\C^n}(\zeta) \leq \int_{B_r(z)\cap X} f(\zeta) dV_X(\zeta)
        \leq C \int_{B_r(0)} \tilde{f}(\zeta) dV_{\C^n}(\zeta),
    \end{equation*}
    where $c$ and $C$ are the constants in \eqref{eq:De2}.
\end{lem}

\begin{proof}
    We claim that
    \begin{equation} \label{eq:distr-functions-comparable}
        c \lambda_{\tilde{f}}(s) \leq \lambda_f(s) \leq C \lambda_{\tilde{f}}(s),
    \end{equation}
    which together with \eqref{eq:estimate0} proves the lemma. 

    To prove the claim, we note first that since $f$ is radial around $z$,
    the level-set $\{ \zeta \in X \mid |f(\zeta)| \leq s \}$ is a union
    of intersections of $X$ with annuli $(B_{r_{1,i}}(z) \setminus B_{r_{2,i}}(z))$.
    The level-set $\{ \zeta \in \C^n \mid |\tilde{f}(\zeta)| \leq s \}$ is a union
    of annuli $(B_{r_{1,i}}(0) \setminus B_{r_{2,i}}(0))$ with the same radii.
    Since
    \begin{equation*}
        c(r_{1,i}^{2n}-r_{2,i}^{2n}) \leq \int_{(B_{r_{1,i}}(z) \setminus B_{r_{2,i}}(z))} dV_X 
        \leq C(r_{1,i}^{2n} - r_{2,i}^{2n})
    \end{equation*}
    by \eqref{eq:De2}, and the fact that $v(r,z)$ is increasing in $r$,
    we then get that \eqref{eq:distr-functions-comparable} holds.
\end{proof}

We then obtain the following important ingredient for our estimates.

\begin{lem}\label{lem:estimate2}
Let $X\subset \C^N$ be an analytic variety of pure dimension $n$, $K\subset X$ a compact subset and $R>0$. Fix also $\alpha\geq 0$.
Then there exists a constant $C_1>0$ such that the following holds:
\begin{eqnarray*}
I(r_1,r_2) := \int_{X \cap \left(B_{r_2}(z)\setminus \o{B_{r_1}(z)}\right)} 
\frac{dV_X(\zeta)}{\|\zeta-z\|^\alpha} \leq C_1 \left\{ 
\begin{array}{ll}
r_2^{2n-\alpha} & \ ,\ \alpha<2n,\\
1+|\log r_1| & \ ,\ \alpha=2n,\\
r_1^{2n-\alpha} & \ ,\ \alpha>2n,
\end{array}\right.
\end{eqnarray*}
for all $z\in K$ and $0<r_1 \leq r_2 \leq R$.
\end{lem}

A proof of Lemma~\ref{lem:estimate2} is obtained by combining the corresponding statement when
$X = \C^n$, \cite{LR}, Lemma~A.1, with Lemma~\ref{lem:comparable}.
Similarly, as it is an elementary calculation that the corresponding integral is bounded when $X = \C^n$,
we obtain the following.

\begin{lem}\label{lem:integral-log}
    Let $X$ and $K$ be as in Lemma~\ref{lem:comparable}.
    Then
    \begin{eqnarray*}
        I(z) := \int_{X\cap B_{1/2}(z)} \frac{dV_X(\zeta)}{\|\zeta-z\|^{2n} \log^2\|\zeta-z\|} &\lesssim& 1
    \end{eqnarray*}
    for all $z\in K$.
\end{lem}

For cut-off estimates, we also need the following, which we again by Lemma~\ref{lem:comparable} can
reduce to the case when $X = \C^n$, and this case follows by a straightforward calculation
(cf., \cite{LR}, Lemma~A.4 for a more general variant).

\begin{lem}\label{lem:integral-loglog}
    Let $X$ and $K$ be as in Lemma~\ref{lem:comparable}, and let
    for any integer $m\geq 0$ let $r_m:=e^{-e^m}$. Then  
\begin{eqnarray*}
I_m(z) := \int_{X\cap \big( B_{r_m}(z) \setminus \overline{B_{r_{m+1}}(z)}   \big)} \frac{dV_X(\zeta)}{\|\zeta-z\|^{2n} \big|\log\|\zeta-z\|\big|} &\lesssim& 1
\end{eqnarray*}
for all $z\in K$ uniformly, i.e., not depending on $m$.
\end{lem}

\smallskip
\subsection{Basic integral estimates on analytic varieties}

We now consider integral estimates for integrands which are not radial,
but which are products of radial functions with different centers.
From Lemma \ref{lem:estimate2}, we can deduce our main basic estimate:

\begin{lem}\label{lem:estimate3}
Let $X\subset \C^N$ be an analytic variety of pure dimension $n$, $D\subset\subset X$ relatively compact and $0 \leq \alpha,\beta <2n$. 
Then there exists a constant $C_2>0$ such that the following holds:
\begin{eqnarray*}
\int_D \frac{dV_X(\zeta)}{\|\zeta-z\|^\alpha \|\zeta-w\|^\beta}
\leq C_2 \left\{
\begin{array}{ll}
1 & \ ,\ \alpha+\beta<2n,\\
\big| \log \|z - w\| \big| &\ ,\ \alpha+\beta=2n,\\
\|z - w\|^{2n-\alpha-\beta} &\ ,\ \alpha+\beta>2n,
\end{array}\right.
\end{eqnarray*}
for all $z, w \in X$ with $z\neq w$.
\end{lem}

Lemma~\ref{lem:estimate3} follows from Lemma~\ref{lem:estimate2} in exactly the same
way as Lemma~A.2 in \cite{LR} follows from Lemma~A.1 in \cite{LR}.

Also needed and a little more sophisticated is the following:

\begin{lem}\label{lem:estimate5}
Let $X\subset \C^N$ be an analytic variety of pure dimension $n$, $D\subset\subset X$ relatively compact, and $K \subset X$ compact,
$0 \leq \alpha \leq 2n$
and $0\leq \beta < 2n$. For any integer $m\geq 0$ let $r_m:= e^{-e^m}$.
Then there exists a constant $C_3>0$, not depending on $m$, such that the following holds:
\begin{eqnarray*}
\int_{D\cap \big(B_{r_{m}}(0) \setminus \overline{B_{r_{m+1}}(0)} \big) } \frac{dV_X(\zeta)}{\|\zeta\|^\alpha \big|\log \|\zeta\|\big| \|\zeta-z\|^\beta}
\leq C_3 \left\{
\begin{array}{ll}
1 & \ ,\ \alpha+\beta\leq 2n,\\
\|z\|^{2n-\alpha-\beta} &\ ,\ \alpha+\beta>2n,
\end{array}\right.
\end{eqnarray*}
for all $z \in K$ with $z \neq 0$.
\end{lem}

\begin{proof}
Let $K' := K \cup \overline{D} \cup \{ 0 \}$, and let $R$ be the diameter of $K'$.
Let $\delta:=\|z\|$. Since $z$ and $0$ belong to $K'$, we get that $\delta \leq R$.
We will apply Lemma \ref{lem:estimate2} several times with $K'$ and $R>0$ as chosen above.

We divide the domain of integration $Y:= D\cap \big(B_{r_{m}}(0) \setminus \overline{B_{r_{m+1}}(0)} \big)$
in three regions $D_1$, $D_2$, $D_3$. Let 
$$D_1:= Y \cap B_{\delta/2}(0)\ \ ,\ \ D_2:= Y \cap B_{\delta/2}(z).$$
Then $\|\zeta-z\| \geq \delta/2$ on $D_1$ and so
\begin{eqnarray*}
\int_{D_1} \frac{dV_X(\zeta)}{\|\zeta\|^\alpha \big|\log \|\zeta\|\big| \|\zeta-z\|^\beta}
&\leq& (\delta/2)^{-\beta} \int_{D_1} \frac{dV_X(\zeta)}{\|\zeta\|^\alpha \big|\log \|\zeta\|\big|}\\
&\lesssim&  (\delta/2)^{-\beta+2n-\alpha}.
\end{eqnarray*}
The last step follows by Lemma \ref{lem:estimate2} if $\alpha<2n$ 
(using $|\log|^{-1} \|\zeta\| \lesssim 1$, and letting $r_1\rightarrow 0$ in Lemma \ref{lem:estimate2}),
and by Lemma~\ref{lem:integral-loglog} if $\alpha=2n$.

\medskip
As $\|\zeta\|\geq \delta/2$ on $D_2$ we have similarly:
\begin{eqnarray*}
\int_{D_2} \frac{dV_X(\zeta)}{\|\zeta\|^\alpha \big| \log \|\zeta\|\big| \|\zeta-z\|^\beta} 
&\leq& (\delta/2)^{-\alpha} \int_{X\cap B_{\delta/2}(z)} \frac{dV_X(\zeta)}{\|\zeta-z\|^\beta} \\
&\leq& C_1 (\delta/2)^{-\alpha+2n-\beta},
\end{eqnarray*}
where we need only Lemma \ref{lem:estimate2} for the last step.

It remains to consider the integral over $Y \setminus (D_1\cup D_2)$.
Here, $\|\zeta-z\|\geq \delta/2$ and that yields:
\begin{eqnarray*}
\|\zeta\| \leq \|\zeta-z\| + \|z\| = \|\zeta- z\| + \delta \leq 3 \|\zeta-z\|.
\end{eqnarray*}
So, we can estimate:
\begin{eqnarray*}
\int_{Y \setminus (D_1\cup D_2)} \frac{dV_X(\zeta)}{\|\zeta\|^\alpha \big| \log\|\zeta\|\big| \|\zeta-z\|^\beta}
 &\leq& 3^{\beta} \int_{Y \cap (B_{R}(0)\setminus \o{B_{\delta/2}(0)})} \frac{dV_X(\zeta)}{\|\zeta\|^{\alpha+\beta}\big|\log\|\zeta\|\big|}\\
&\lesssim& 3^{\beta} C_1  \left\{ 
\begin{array}{ll}
R^{2n-\alpha-\beta} & \ ,\ \alpha+\beta \leq 2n,\\
(\delta/2)^{2n-\alpha-\beta} & \ ,\ \alpha+\beta \geq 2n.
\end{array}\right.
\end{eqnarray*}
For the last step, we use Lemma \ref{lem:estimate2} if $\alpha+\beta\neq 2n$,
and Lemma \ref{lem:integral-loglog} otherwise.

The assertion follows easily from this statement in combination with the estimates for the integration over $D_1$ and $D_2$.
\end{proof}

For $C^\alpha$-estimates, we will use the following variant of Lemma~\ref{lem:estimate2}.

\begin{lem}\label{lem:estimate8}
Let $X\subset \C^N$ be an analytic variety of pure dimension $n$, $K\subset X$ a
compact subset and $R>0$. Fix also $0 \leq \alpha < 2n$.
Then there exists a constant $C_4>0$ such that:
\begin{eqnarray*}
I_r(z) := \int_{X \cap B_{r}(z)} 
\frac{dV_X(\zeta)}{\|\zeta-w\|^\alpha} \leq C_4
r^{2n-\alpha}
\end{eqnarray*}
for all $z\in K$, $w \in X$ and $0\leq r \leq R$.
\end{lem}

\begin{proof}
    We first consider the case when $B_r(z) \cap B_r(w) = \emptyset$.
    Then, $\|\zeta-w\| > r$ on $B_r(z)$, so
    \begin{equation*}
        I_r(z) \leq \frac{1}{r^\alpha}\int_{X \cap B_{r}(z)} dV_X(\zeta)
        \leq C_1 r^{2n-\alpha}
    \end{equation*}
    by Lemma~\ref{lem:estimate2}.

    It remains to consider the case when $B_r(z) \cap B_r(w) \neq \emptyset$.
    Then, $B_r(z) \subseteq B_{3r}(w)$.
    Hence, again by Lemma~\ref{lem:estimate2},
    \begin{equation*}
        I_r(z) \leq \int_{X \cap B_{3r}(w)} \frac{dV_X(\zeta)}{\|\zeta-w\|^\alpha}
        \leq C_1 (3r)^{2n-\alpha}.
    \end{equation*}
\end{proof}

\section{$C^\alpha$- and $L^p$-forms on an analytic variety}\label{sec:lp-forms}

Our main results deal with $C^\alpha$- and $L^p$-forms on an analytic variety, so we precise
here its meaning, and remind of some basic results about such forms.
Let $X \subseteq \C^N$ be an analytic variety of pure dimension $n$, and let $D \subset\subset X$ be an open set.
Let $1\leq p\leq \infty$. Since $D^* = D \cap \Reg X$ is a submanifold of some open subset of $\C^N$, it inherits a Hermitian metric,
and we say that a $(0,q)$-form $\varphi$ on $D$ is in $L^p_{0,q}(D)$ if $\varphi|_{D^*}$ is in $L^p_{0,q}(D^*)$ with respect
to the induced volume form $dV_X$. Note that as remarked before, $\Sing X$ is a null-set with respect to $dV_X$,
so it does not matter if we consider $L^p$-forms on $D$ or $D^*$.

When we consider an $L^p$-differential form as input into an integral operator,
it will be convenient to represent it in a certain ``minimal'' manner.
If $\varphi$ is a $(0,q)$-form on $D$, then by \cite[Lemma~2.2.1]{RuThesis}, we can
write $\varphi$ uniquely in the form
\begin{equation}\label{eq:minrep}
    \varphi = \sum_{|I|=q} \varphi_I d\bar{z}_I,
\end{equation}
where
\begin{equation*}
    |\varphi|^2 (z)= 2^q \sum |\varphi_I|^2(z)
\end{equation*}
in each regular point $z\in D^*$.
The constants here stem from the fact that $|d\overline{z_j}|=\sqrt{2}$ in $\C^n$.
In particular, we then get that $\varphi \in L^p_{0,q}(D)$ if and only if $\varphi_I \in L^p(D)$ for all $I$.
If one has an arbitrary representation of $\varphi$ of the form \eqref{eq:minrep}, then
\begin{equation} \label{eq:non-minrepr}
    |\varphi|^2 (z) \leq 2^q \sum |\varphi_I|^2(z),
\end{equation}
and then, $\varphi \in L^p_{0,q}(D)$ if $\varphi_I \in L^p(D)$ for all $I$.

\medskip
For $0 \leq \alpha < 1$, we say that a $(0,q)$-form $\varphi$ is $C^\alpha$ at a point $z\in D$
if there is a representation \eqref{eq:minrep} such that all the coefficients $\varphi_I$ are
$C^\alpha$, i.e., H\"older continuous with exponent $\alpha$, at the point $z$. 
We denote by $C^\alpha_{0,q}(D)$ the vector space of $C^\alpha$-forms
on the domain $D$. $C^\alpha(D)$ is a Fr\'echet space with the usual metric, and we give $C^\alpha_{0,q}(D)$
the largest topology making the mapping
\begin{eqnarray*}
\bigoplus_{|I|=q} C^\alpha(D) \rightarrow C^\alpha_{0,q}(D)\ \ \ ,\ \ \ \big( \varphi_I \big)_I \mapsto \sum_{|I|=q} \varphi_I d\bar{z}_I
\end{eqnarray*}
continuous.
For $\alpha = 1$, we denote the Lipschitz continuous functions by $C^{0,1}(D)$,
in order to avoid conflict of notation with continuously differentiable functions.

Using the minimal representation \eqref{eq:minrep}, and the inequality \eqref{eq:non-minrepr}
for not necessarily minimal representations, the following lemma follows immediately.

\begin{lem} \label{lma:Lpformsfunctions}
If $\mathcal{K}$ is an integral operator mapping $(0,q)$-forms in $\zeta$ to $(0,q-1)$-forms in $z$,
defined by an integral kernel
\begin{equation*}
    K(\zeta,z) = \sum_{|L|=n,|I|=q-1,|J|={n-q}} K_{I,J,L}(\zeta,z) d\overline{z}_I \wedge {d\overline{\zeta}_J} \wedge d\zeta_L,
\end{equation*}
then $\mathcal{K}$ is a bounded linear map
$L_{0,q}^p(D') \to L_{0,q-1}^p(D)$ if
\begin{equation*}
    f(\zeta) \mapsto \int_{D'} K_{I,J,L}(\zeta,z) f(\zeta) dV_X(\zeta)
\end{equation*}
is a bounded linear map $L^p(D') \to L^p(D)$, and a continuous linear map
$L_{0,q}^\infty(D') \to C_{0,q-1}^\alpha(D)$ if
\begin{equation*}
    f(\zeta) \mapsto \int_{D'} K_{I,J,L}(\zeta,z) f(\zeta) dV_X(\zeta)
\end{equation*}
is a continuous linear map $L^\infty(D') \to C^\alpha(D)$.
\end{lem}

\medskip
\section{Estimates for integral operators with isotropic isolated poles
on varieties with arbitrary singularities}
\label{sec:Iestimates}

Let $X\subset \C^N$ be an analytic variety of pure dimension $n$.
We will consider properties of the integral kernel
\begin{eqnarray}\label{eq:kernel1}
k_\gamma(\zeta,z):= \frac{\|z\|^\gamma}{\|\zeta\|^\gamma \|\zeta-z\|^{2n-1}}
\end{eqnarray}
on $X$ for $0\leq \gamma < 2n$.

\medskip
\subsection{$L^p$-mapping properties}

Our basic estimate, Lemma \ref{lem:estimate3}, allows to study $L^p$-mapping properties of integral operators
given by the kernels $k_\gamma(\zeta,z)$ defined in \eqref{eq:kernel1} by the use of generalized Young inequalities.

\begin{thm}\label{thm:lp-estimate}
Let $D\subset\subset X$ be a bounded domain in $X$. Let $0\leq \gamma<2n$.
Then the integral operator
\begin{eqnarray*}
f \mapsto {\bf T}(f)(z):=\int_D f(\zeta) k_\gamma(\zeta,z) dV_X(\zeta)
\end{eqnarray*}
defines  a bounded linear operator ${\bf T}: L^p(D) \rightarrow L^p(D)$ for all $\frac{2n}{2n-\gamma} < p \leq \infty$.
\end{thm}

\begin{proof}
Let us first consider the case $p<\infty$. Choose 
    $$p^*:=p/(p-1).$$
    So, $1/p+1/p^*=1$. Moreover, we get:
    \begin{eqnarray*}
    \gamma p^* < 2n \Leftrightarrow 1/p^* > \gamma/2n \Leftrightarrow 1-\gamma/2n > 1/p \Leftrightarrow p> \frac{2n}{2n-\gamma},
    \end{eqnarray*}
   so that actually $\gamma p^*<2n$ by the assumption on $p$.
   
       We want to show that the $L^p$-norm of ${\bf T}f$ is finite, and
    we begin by estimating and decomposing, and using the H\"older inequality (with $1/p+1/p^*=1$) in the following way:
   
    \begin{align*}
        I := &\int_D  \left| \int_D f(\zeta)k_\gamma(\zeta,z) dV_X(\zeta) \right|^p dV_X(z) \\ \leq
        &\int_D\left( \int_D \left(\frac{|f(\zeta)|^p}{\|\zeta-z\|^{2n-1}} \right)^{1/p}
        \left( \frac{\|z\|^{p^*\gamma}}{\|\zeta\|^{p^*\gamma}\|\zeta-z\|^{2n-1}}\right)^{1/p^*}  dV_X(\zeta) \right)^{p} dV_X(z) \\ \leq
        &\int_D \int_D \frac{|f(\zeta)|^p}{\|\zeta-z\|^{2n-1}}  dV_X(\zeta) 
        \left(\int_D \frac{\|z\|^{p^*\gamma}}{\|\zeta\|^{p^*\gamma}\|\zeta-z\|^{2n-1}} dV_X(\zeta)\right)^{p/p^*}  dV_X(z).
    \end{align*}
    Inserting
     \begin{eqnarray*}   
    \int_D \frac{\|z\|^{p^*\gamma}}{\|\zeta\|^{p^*\gamma}\|\zeta-z\|^{2n-1}} dV_X(\zeta) &\lesssim& 1
    \end{eqnarray*}
    which we get by use of Lemma  \ref{lem:estimate3} (recall that $p^*\gamma<2n$),
    and applying the Fubini theorem gives:
    \begin{eqnarray*}
    I &\lesssim& \int_D |f(\zeta)|^p \int_D \frac{dV_X(z)}{\|\zeta-z\|^{2n-1}} dV_X(\zeta)\\
    &\lesssim& \int_D |f(\zeta)|^p dV_X(\zeta) = \|f\|_{L^p(D)}^p,
    \end{eqnarray*}
    where we have applied Lemma \ref{lem:estimate3} once more (for the integral in $z$).
    
    \bigskip    
    It remains to consider the case $p=\infty$ which is even simpler:
    \begin{eqnarray*}
    \left| \int_D f(\zeta) k_\gamma(\zeta,z) dV_X(\zeta) \right| &\leq& \|f\|_\infty \int k_\gamma(\zeta,z) dV_X(\zeta) \lesssim \|f\|_\infty
         \end{eqnarray*}
    by use of Lemma \ref{lem:estimate3} (with the assumption that $\gamma<2n$).
 \end{proof}

\begin{lem}\label{lem:cpt-cutoff-estimate}
Let $D\subset\subset D' \subset\subset X$ be bounded domains in $X$. Let $0\leq \gamma<2n$,
and let for $j > 0$,
\begin{equation*}
    k_{j,\gamma}(\zeta,z) := \left\{ \begin{array}{cc} 0 & \text{ if } k_\gamma(\zeta,z) > j \\ k_\gamma(\zeta,z) & \text{ otherwise } \end{array} \right..
\end{equation*}
Let
\begin{eqnarray*}
    f \mapsto {\bf T}_j(f)(z):=\int_{D'} f(\zeta) k_{j,\gamma}(\zeta,z) dV_X(\zeta)
\end{eqnarray*}
and
\begin{eqnarray*}
    f \mapsto {\bf T}(f)(z):=\int_{D'} f(\zeta) k_\gamma(\zeta,z) dV_X(\zeta).
\end{eqnarray*}
Then $\|{\bf T}_j - {\bf T}\| \to 0$ as bounded linear operators $L^p(D') \rightarrow L^p(D)$ for all $\frac{2n}{2n-\gamma} < p < \infty$.
\end{lem}

\begin{proof}
    The proof follows in a way similar to the proof of Theorem~\ref{thm:lp-estimate}. Take $f \in L^p(D')$.
    Following that proof, one gets that
    \begin{equation} \label{eq:TminusTj}
        \|({\bf T}_j - {\bf T})f\|^p_{L^p(D)} \leq
        \int_D \int_{D'} \frac{\|f(\zeta)\|^p dV_X(\zeta)}{\|\zeta-z\|^{2n-1}} I_j(z) dV_X(z),
    \end{equation}
    where
    \begin{equation*}
        I_j(z) = \left( \int_{ D_j } \frac{\|z\|^\gamma dV_X(\zeta)}{\|\zeta\|^{\gamma} \|\zeta-z\|^{2n-1} } \right)^{p/p^*}
    \end{equation*}
    and $D_j := \{ \zeta \in D' \mid k_\gamma(\zeta,z) > j \}$.
    We note that
    \begin{align}
    \notag D_j \subseteq &\{ \zeta \in D' \mid \|z\|^\gamma/\|\zeta\|^\gamma > \sqrt{j} \} \cup 
        \{ \zeta \in D' \mid 1/\|\zeta-z\|^{2n-1} > \sqrt{j} \} \subseteq \\
        \label{eq:Djballs}    \subseteq & D' \cap (B_{\|z\|/j^{1/(2\gamma)}}(0) \cup B_{1/j^{1/(4n-2)}}(z))
    \end{align}
    when $\gamma > 0$. If $\gamma = 0$, we just interpret the first ball to be empty. We now claim that
    there exist $C_j$ such that $I_j(z) \leq C_j \to 0$.
    To see this, we note that the integrand in $I_j$ is bounded by $M_1/\|\zeta\|^{2n-1} + M_2/\|\zeta-z\|^{2n-1}$.
    By Lemma~\ref{lem:estimate8}, the integral of both these terms on the balls in \eqref{eq:Djballs} tends
    to $0$ since the radii tend to $0$, proving the claim.
    To conclude, from \eqref{eq:TminusTj}, similarly to the proof of Theorem~\ref{thm:lp-estimate},
    we get that
    \begin{equation*}
        \|({\bf T}_j-{\bf T})f\|_{L^p(D)} \leq C C_j \|f\|_{L^p(D')},
    \end{equation*}
    where $C$ is independent of $j$ and $f$.
\end{proof}

\medskip
\subsection{Continuity estimates}

\begin{thm}\label{thm:C0-estimate}
Let $D\subset\subset X$ be a bounded domain in $X$. Let $\gamma\in \Z$, $0\leq \gamma<2n$,
and let
\begin{equation*}
    \tilde{k}_\gamma(\zeta,z) := \frac{\|z\|^\gamma}{\|\zeta\|^\gamma}\frac{\overline{\zeta_i-z_i}}{\|\zeta-z\|^{2n}},
\end{equation*}
for some $i \in \{1,\dots,n\}$.
Then the integral operator
\begin{eqnarray*}
    f \mapsto {\bf T}(f)(z):=\int_D f(\zeta) \tilde{k}_\gamma(\zeta,z) dV_X(\zeta)
\end{eqnarray*}
defines a compact continuous linear operator ${\bf T}: L^\infty(D) \rightarrow C^\alpha(D)$,
where $0 \leq \alpha < 1$.
\end{thm}

If $\gamma = 0$, then a standard proof from the case $X=\C^n$,
as for example \cite[Proposition~III.2.1]{LT}, works, by using Lemma~\ref{lem:estimate3}.
We will adapt this proof to work also for $\gamma > 0$.

\begin{proof}
    Since
    \begin{equation*}
        |{\bf T}(f)(z) - {\bf T}(f)(w)| \leq \|f\|_{L^\infty(D)} \int_D | \tilde{k}_\gamma(\zeta,z) - \tilde{k}_{\gamma}(\zeta,w)| dV_X(\zeta),
    \end{equation*}
    in order to prove the continuity as a map $L^\infty(D) \rightarrow C^\alpha(D)$
    it is enough to prove that for $\alpha < 1$ fixed,
    \begin{equation} \label{eq:calpha}
        \int_D | \tilde{k}_\gamma(\zeta,z) - \tilde{k}_{\gamma}(\zeta,w) | dV_X(\zeta) \lesssim \|z-w\|^\alpha.
    \end{equation}
    for $z,w \in D$.
    In order to do this, we let $r := \|z-w\|/2$, and partition $D$ into
    \begin{align*}
        W_1 &:= D \cap B_r(z)\text{, } W_2 := D\cap B_r(w) \text{, } W_3 := (D\setminus (W_1 \cup W_2)) \cap B_r(0)
        \text{ and } \\ W_4 &:= D\setminus (W_1 \cup W_2 \cup W_3),
    \end{align*}
    and prove the inequality for the integrals over each of the $W_i$'s.
    Using that $\|z\| \leq \|\zeta\| + \|\zeta-z\|$, we get that
    \begin{align*}
        &\int_{W_1} | \tilde{k}_\gamma(\zeta,z) - \tilde{k}_{\gamma}(\zeta,w) | dV_X(\zeta) \lesssim \\
        \sum_{k=0}^{2n-1} & \int_{B_r(z) \cap X} \frac{1}{\|\zeta\|^k \|\zeta-z\|^{2n-k-1}} + \frac{1}{\|\zeta\|^k \|\zeta-w\|^{2n-k-1}} dV_X(\zeta) \lesssim \\
        &\int_{B_r(z) \cap X} \max \left\{ \frac{1}{\|\zeta\|^{2n-1}}, \frac{1}{\|\zeta-z\|^{2n-1}} \right\} +
        \max \left\{ \frac{1}{\|\zeta\|^{2n-1}}, \frac{1}{\|\zeta-w\|^{2n-1}} \right\} dV_X(\zeta) \lesssim \\
        &\int_{B_r(z) \cap X} \frac{1}{\|\zeta\|^{2n-1}} + \frac{1}{\|\zeta-z\|^{2n-1}}  +
        \frac{1}{\|\zeta-w\|^{2n-1}}  dV_X(\zeta) \lesssim r
    \end{align*}
    where the last inequality follows by Lemma~\ref{lem:estimate8}.
    By symmetry, we get the same estimate for the integral on $W_2$.
    In the same way as for the calculation on $W_1$, but using that on $W_3$, $\|\zeta-z\| \geq r$, and $\|\zeta-w\| \geq r$,
    we get that
    \begin{align*}
        \int_{W_3} | \tilde{k}_\gamma(\zeta,z) - \tilde{k}_{\gamma}(\zeta,w) | dV_X(\zeta) \lesssim 
        \sum_{k=0}^{2n-1} & \frac{1}{r^{2n-k-1}} \int_{B_r(0) \cap X} \frac{1}{\|\zeta\|^k} dV_X(\zeta)
         \lesssim r,
    \end{align*}
    where we used Lemma~\ref{lem:estimate2} for the last inequality.

    Finally, we consider the integral on $W_4$. By possibly switching the roles of $z$ and $w$,
    we can assume that $\|w\| \leq \|z\|$.
    First, we write
    \begin{align*}
        |\tilde{k}_\gamma(\zeta,z)-\tilde{k}_\gamma(\zeta,w)| \leq
        \frac{|\|z\|^\gamma-\|w\|^\gamma|}{\|\zeta\|^\gamma} \|\tilde{k}_0(\zeta,z)\|
        +\frac{\|w\|^\gamma}{\|\zeta\|^\gamma} |\tilde{k}_0(\zeta,z) - \tilde{k}_0(\zeta,w)|.
    \end{align*}
    If we consider the first term, and use the reverse triangle inequality $|\|z\|-\|w\|| \leq \|z-w\|$,
    $a^\gamma-b^\gamma=(a-b)(a^{\gamma-1} + \dots + b^{\gamma-1})$, $\max(a,b) \leq a+b$ if $a,b \geq 0$,
    and the assumption that $\|w\|\leq \|z\|$, we get that
    \begin{align*}
        \frac{|\|z\|^\gamma-\|w\|^\gamma|}{\|\zeta\|^\gamma} \|\tilde{k}_0(\zeta,z)\|
        \leq &\|z-w\| \sum_{\ell=0}^{\gamma-1} \frac{1}{\|\zeta\|^{\gamma-\ell}\|\zeta-z\|^{2n-(\gamma-\ell)}} \\
        \lesssim &\|z-w\| \left(\frac{1}{\|\zeta\|^{2n}} + \frac{1}{\|\zeta-z\|^{2n}}\right).
    \end{align*}
    Since $W_4 \subseteq B_R(0) \setminus B_r(0)$, and $W_4 \subseteq B_R(z) \setminus B_r(z)$,
    for $R \gg 0$, we get by Lemma~\ref{lem:estimate2} that
    \begin{equation*}
        \int_{W_4}\frac{|\|z\|^\gamma-\|w\|^\gamma|}{\|\zeta\|^\gamma} \|\tilde{k}_0(\zeta,z)\| \leq \|z-w\|(1+ |\log \|z-w\||).
    \end{equation*}

    Finally, as in the proof of \cite[Lemma~III.2.2]{LT},
    \begin{equation*}
        |\tilde{k}_0(\zeta,z)-\tilde{k}_0(\zeta,w)| \lesssim \|z-w\| \max \left\{ \frac{1}{\|\zeta-z\|^{2n}},\frac{1}{\|\zeta-w\|^{2n}}\right\} 
        \leq \|z-w\| \left(\frac{1}{\|\zeta-z\|^{2n}}+\frac{1}{\|\zeta-w\|^{2n}}\right).
    \end{equation*}
    Thus, using that $\|w\| \leq \|z\| \leq \|\zeta\| + \|\zeta-z\|$, and $\|w\| \leq \|\zeta\| + \|\zeta-w\|$,
    we get that
    \begin{align*}
        \frac{\|w\|^\gamma}{\|\zeta\|^\gamma} |\tilde{k}_0(\zeta,z) - \tilde{k}_0(\zeta,w)| \lesssim & 
        \|z-w\| \sum_{\ell=0}^{\gamma}\left( \frac{1}{\|\zeta\|^\ell \|\zeta-z\|^{2n-\ell}}
        + \frac{1}{\|\zeta\|^\ell \|\zeta-w\|^{2n-\ell}}\right) \\
    \lesssim & \|z-w\| \left( \frac{1}{\|\zeta\|^{2n}} + \frac{1}{\|\zeta-z\|^{2n}} + \frac{1}{\|\zeta-w\|^{2n}} \right)
    \end{align*}
    Since $W_4$ is contained in $B_R(0)\setminus B_r(0)$, $B_R(z)\setminus B_r(z)$ and $B_R(w) \setminus B_r(w)$
    if $R \gg 0$, we get by Lemma~\ref{lem:estimate2} that
    \begin{equation*}
        \int_{W_4} \frac{\|w\|^\gamma}{\|\zeta\|^\gamma} |\tilde{k}_0(\zeta,z) - \tilde{k}_0(\zeta,w)| \lesssim r(1+ |\log r|).
    \end{equation*}

    Combining the estimates for the integrals of the left-hand side of \eqref{eq:calpha} on $W_1,W_2,W_3$ and $W_4$,
    we get that the integral on $D$ is bounded by some constant times $r (1 + |\log r|)$, and since $r = \|z-w\|/2$,
    we get that \eqref{eq:calpha} holds for any $\alpha < 1$.

    Since \eqref{eq:calpha} holds uniformly for $z,w$ in $D$, if $\{\varphi_j\}$ is a uniformly bounded sequence in $L^\infty(D')$,
    then $\{{\bf T}(\varphi_j)\}$ is equicontinuous in the $C^\alpha(\overline{D})$-norm, and thus,
    ${\bf T}$ is compact by the Arzel\`a-Ascoli theorem.
\end{proof}

\medskip
\subsection{Estimates for cut-off and approximation procedures}

In order to prove $\dq$-homotopy formulas, we will need to approximate $L^p$-forms
in an appropriate way by smooth forms. For this purpose, we require
the following cut-off estimate for the integral kernels $k_\gamma(\zeta,z)$.

\begin{thm}\label{thm:lp-estimate2}
Let $D\subseteq D' \subset\subset X$ be bounded domains in $X$.
Let $\gamma\in\Z$, $0\leq \gamma <  2n-1$, and let $\frac{2n}{2n-(\gamma+1)} \leq p < \infty$. 
For any integer $m\geq 0$ let $r_m:=e^{-e^m}$.
Then the integral operators
\begin{eqnarray*}
f \mapsto {\bf T}_m(f)(z):=\int_{D'\cap \big(B_{r_m(0)}\setminus\overline{B_{r_{m+1}}(0)}\big)} 
f(\zeta) \frac{k_\gamma(\zeta,z)}{\|\zeta\| \big| \log \|\zeta\|\big|} dV_X(\zeta)
\end{eqnarray*}
define bounded linear operators ${\bf T}_m: L^p(D') \rightarrow L^p(D)$ such that
\begin{eqnarray*}
{\bf T}_m f &\rightarrow& 0 \ \ \mbox{ in } \ L^p(D)
\end{eqnarray*}
for $m\rightarrow \infty$.
\end{thm}

\begin{proof}
    To simplify the notation, let $D_m:=D' \cap \big(B_{r_m(0)}\setminus\overline{B_{r_{m+1}}(0)}\big)$.
    As in the proof of Theorem \ref{thm:lp-estimate}, we use the H\"older inequality with $1/p+1/p^*=1$ as follows:
    \begin{align*}
        I_m := &\int_D  \left| \int_{D_m} f(\zeta) \frac{k_\gamma(\zeta,z)}{\|\zeta\|\big| \log\|\zeta\|\big|} dV_X(\zeta) \right|^p dV(z) \\ =
        &\int_D \left(\int_{D_m} \left(\frac{|f(\zeta)|^p}{\big|\log \|\zeta\|\big|  \|\zeta-z\|^{2n-1}} \right)^{1/p} \right.\\        
        & \cdot \left. \left( \frac{\|z\|^{p^*\gamma}}{\|\zeta\|^{p^*(\gamma+1)} \big| \log \|\zeta\|\big| \|\zeta-z\|^{2n-1}}\right)^{1/p^*}  
        dV_X(\zeta) \right)^p dV_X(z) \\ \leq
        &\int_D \left( \int_{D_m} \frac{|f(\zeta)|^p}{\big|\log\|\zeta\|\big| \|\zeta-z\|^{2n-1}}  dV_X(\zeta) \right)\\
        & \cdot \left( \int_{D_m} \frac{\|z\|^{p^*\gamma}}{\|\zeta\|^{p^*(\gamma+1)}\big|\log\|\zeta\|\big|\|\zeta-z\|^{2n-1}} dV_X(\zeta)\right)^{p/p^*} dV_X(z).
    \end{align*}
    As in the proof of Theorem~\ref{thm:lp-estimate}, if $p \geq \frac{2n}{2n-(\gamma+1)}$, then $p^*(\gamma+1) \leq 2n$.
    Inserting
     \begin{eqnarray*}   
    \int_{D_m} \frac{\|z\|^{p^*\gamma}}{\|\zeta\|^{p^*(\gamma+1)}\big|\log\|\zeta\|\big|\|\zeta-z\|^{2n-1}} dV_X(\zeta) 
    &\lesssim& \|z\|^{p^*\gamma+2n-(2n-1)-p^*(\gamma+1)} = \|z\|^{1-p^*},
    \end{eqnarray*}
    which we get by use of Lemma  \ref{lem:estimate5} (since $p^*(\gamma+1) \leq 2n$),
    and applying the Fubini theorem gives (keep in mind that $\frac{1}{p^*}(1-p^*)=- \frac{1}{p}$):
    \begin{eqnarray*}
    I_m &\lesssim& \int_{D_m} \frac{|f(\zeta)|^p}{\big|\log\|\zeta\|\big|} \int_D \frac{dV_X(z)}{\|z\|\|\zeta-z\|^{2n-1}} dV_X(\zeta)\\
    &\lesssim& \int_{D_m} |f(\zeta)|^p dV_X(\zeta) = \|f\|_{L^p(D_m)}^p,
    \end{eqnarray*}
    where we have applied Lemma \ref{lem:estimate3} once more. 
    
    But now $\|f\|_{L^p(D_m)} \rightarrow 0$ for $k\rightarrow \infty$ because the domain of integration vanishes and $p<\infty$
    (see e.g. \cite{Alt}, A.1.16.2).
 \end{proof}

\medskip
\section{Approximation by smooth forms}\label{sec:cut-off}

\subsection{Cut-off functions}\label{ssec:cut-off}

We will use the following cut-off functions to approximate forms
by forms with support away from the singularity in different situations.

As in \cite{PS}, Lemma 3.6, let $\rho_k: \R\rightarrow [0,1]$, $k\geq 1$, be smooth cut-off functions
satisfying 
$$\rho_k(x)=\left\{\begin{array}{ll}
1 &,\  x\leq k,\\
0 &,\  x\geq k+1,
\end{array}\right.$$
and $|\rho_k'|\leq 2$. Moreover, let $r: \R\rightarrow [0,1/2]$ be a smooth increasing function such that
$$r(x)=\left\{\begin{array}{ll}
x &,\ x\leq 1/4,\\
1/2 &,\ x\geq 3/4,
\end{array}\right.$$
and $|r'|\leq 1$.
As cut-off functions we will use 
\begin{eqnarray*}
\mu_k(\zeta):=\rho_k\big(\log(-\log r(\|\zeta\|))\big)
\end{eqnarray*} 
on $X$. Note that
\begin{eqnarray}\label{eq:cutoff2}
\big| \dq \mu_k(\zeta)\big| \lesssim \frac{\chi_k(\|\zeta\|)}{\|\zeta\| \big| \log\|\zeta\|\big|},
\end{eqnarray}
where $\chi_k$ is the characteristic function of $[e^{-e^{k+1}}, e^{-e^k}]$.

\begin{lem}\label{lem:cut-off}
Let $X$ be an analytic variety of pure dimension $n$ in $\C^N$, $D\subset\subset X$ an open subset and let $\varphi\in L^p_{0,q}(D)$
with $\dq_w \varphi\in L^r_{0,q+1}(D)$, where $\frac{2n}{2n-1}\leq p \leq \infty$ and $1\leq r \leq \infty$.
Let
\begin{eqnarray*}
\varphi_k &:=& \mu_k \varphi
\end{eqnarray*}
and define $1\leq \lambda \leq 2n$ by the relation
\begin{eqnarray}\label{defn:lambda}
\frac{1}{\lambda} &=& \frac{1}{p} + \frac{1}{2n}.
\end{eqnarray}
Then
\begin{eqnarray*}
\varphi_k \rightarrow \varphi  &\mbox{ in }& L^p_{0,q}(D),\\
\dq \varphi_k \rightarrow \dq_w \varphi &\mbox{ in }& L^\gamma_{0,q+1}(D),
\end{eqnarray*}
where $\gamma=\min\{\lambda, r\}$.
\end{lem}

\begin{proof}
Is is easy to see by Lebesgue's theorem on dominated convergence that
\begin{eqnarray*}
\varphi_k = \mu_k \varphi \rightarrow \varphi \ \ \ \mbox{ in } L^p_{0,q}(D)\ \ \ ,\ \ \ 
\mu_k \dq_w \varphi \rightarrow \dq_w \varphi \ \ \ \mbox{ in } L^r_{0,q+1}(D).
\end{eqnarray*}
It just remains to show that
\begin{eqnarray*}
\dq\mu_k \wedge \varphi &\rightarrow& 0 \ \ \ \mbox{ in } L^\gamma_{0,q+1}(D).
\end{eqnarray*}
So, we use the H\"older inequality (with the relation \eqref{defn:lambda}) to estimate
\begin{eqnarray*}
\|\dq\mu_k \wedge \varphi\|_{L^\gamma} &\leq& \|\varphi\|_{L^p} \|\dq\mu_k\|_{L^{2n}}.
\end{eqnarray*}
But by use of \eqref{eq:cutoff2} we get
\begin{eqnarray*}
\|\dq\mu_k\|^{2n}_{L^{2n}} \leq \int_{X\cap \supp \chi_k} \frac{dV_X(\zeta)}{\|\zeta\|^{2n} \log^{2n} \|\zeta\|} 
\leq \int_{X\cap \supp \chi_k} \frac{dV_X(\zeta)}{\|\zeta\|^{2n} \log^{2} \|\zeta\|}
&\rightarrow& 0
\end{eqnarray*}
for $k\rightarrow 0$ because the integrand is integrable over bounded domains in $X$ by Lemma \ref{lem:integral-log}
and the domain of integration vanishes as $k\rightarrow\infty$ (see e.g. \cite{Alt}, A.1.16.2).
\end{proof}

\medskip
\subsection{On the domain of $\dq_s$}

\begin{lem}\label{lem:cut-off2}
Let $X$ be an analytic variety of pure dimension $n$ in $\C^N$ with an isolated singularity at the origin,
$D\subset\subset X$ an open subset with smooth boundary. Let $1\leq p \leq 2n$ and
let $\varphi\in L^p_{0,q}(D)$ such that $\varphi \in \Dom \dq_w^{(p)}$, i.e., $\dq_w \varphi \in L^p_{0,q+1}(D)$.

Then $\varphi\in \Dom\dq_s^{(p)}$ exactly if there
exists a sequence of bounded forms $\varphi_j\in L^\infty_{0,q}(D)$, $\varphi_j\in\Dom\dq_w^{(p)}$,
such that
\begin{eqnarray}\label{eq:appr1}
\varphi_j &\rightarrow& \varphi \ \ \ \mbox{ in } L^p_{0,q}(D),\\
\dq_w \varphi_j &\rightarrow& \dq_w \varphi \ \ \ \mbox{ in } L^p_{0,q+1}(D).\label{eq:appr2}
\end{eqnarray}
\end{lem}

\begin{proof}
Assume that $\varphi\in \Dom\dq_s^{(p)}$. So there exists a sequence of forms $\varphi_j\in C^\infty_{0,q}(D)$, $\varphi_j\in\Dom\dq_w^{(p)}$,
with support away from the isolated singularity at the origin and
such that \eqref{eq:appr1}, \eqref{eq:appr2} holds.
By smoothing with Dirac sequences (on the smooth manifold $X^*$),
we can assume that the $\varphi_j$ are bounded (actually even $\varphi_j\in C^\infty_{0,q}(\overline{D})$).
More precisely, because it has support away from the singularity, a fixed $\varphi_j$ can be approximated
in the graph norm \eqref{eq:appr1}, \eqref{eq:appr2} by forms in $C^\infty_{0,q}(\overline{D})$
by the procedure described in \cite{Alt}, Lemma A 6.7.

For the converse statement, let $\epsilon>0$. Choose $\varphi_j$ such that
\begin{eqnarray}\label{eq:appr01}
\|\varphi-\varphi_j\|_{L^p(D)} < \epsilon/3  &\mbox{ and }& \|\dq_w \varphi - \dq_w \varphi_j\|_{L^p(D)} < \epsilon/3.
\end{eqnarray}
Now use the fact that $\varphi_j$ is bounded and
Lemma \ref{lem:cut-off} (with $p=\infty$ and $\lambda=2n$) to choose $k\geq 0$
such that
\begin{eqnarray}\label{eq:appr02}
\|\varphi_j-\mu_k \varphi_j\|_{L^p(D)} < \epsilon/3  &\mbox{ and }& \|\dq_w \varphi_j - \dq_w (\mu_k \varphi_j)\|_{L^p(D)} < \epsilon/3.
\end{eqnarray}
Now then, $\mu_k \varphi_j$ has support away from the isolated singularity at the origin,
so we can use the procedure from above (\cite{Alt}, Lemma A 6.7)
to find a smooth form $\varphi_\epsilon \in C^\infty_{0,q}(\overline{D})$
with support away from the origin such that
\begin{eqnarray}\label{eq:appr03}
\|\mu_k \varphi_j- \varphi_\epsilon\|_{L^p(D)} < \epsilon/3  &\mbox{ and }& \|\dq_w (\mu_k \varphi_j) - \dq_w \varphi_\epsilon\|_{L^p(D)} < \epsilon/3.
\end{eqnarray}

Combining \eqref{eq:appr01}, \eqref{eq:appr02} and \eqref{eq:appr03},
we have seen that there exists for any $\epsilon>0$ a smooth form $\varphi_\epsilon$ with support
away from the singularity such that
\begin{eqnarray*}
\|\varphi-\varphi_\epsilon\|_{L^p(D)} < \epsilon  &\mbox{ and }& \|\dq_w \varphi - \dq_w \varphi_\epsilon\|_{L^p(D)} < \epsilon.
\end{eqnarray*}
This means nothing else but $\varphi\in \Dom\dq_s^{(p)}$.
\end{proof}

\medskip

\section{The Andersson-Samuelsson integral operator for affine cones over smooth projective complete intersections}
\label{sec:main}

\smallskip
\subsection{The Koppelman integral operator for a reduced complete intersection}

For convenience of the reader, let us recall shortly the definition of the Koppelman integral operators from \cite{AS}
in the situation of a reduced complete intersection $X \subseteq \C^N$ of dimension $n = N-\nu$,
defined by $X = \{ \zeta \in \C^N \mid f(\zeta) = 0 \}$, for some tuple $f = (f_1,\dots,f_\nu)$ of holomorphic functions on $\C^N$.
Let $\Omega \subset\subset \Omega' \subset \subset \C^N$ be two strictly pseudoconvex domains,
and let $D := X \cap \Omega$ and $D' := X \cap \Omega'$.

Let $\omega_X$ be a structure form on $X$ (see \cite{AS}, Section 3).
The structure form $\omega_X$ is essentially the pull-back of
\begin{equation}\label{eq:structure-form}
    \frac{\sum_I \overline{\det \frac{\partial f}{\partial \zeta_I}} \widehat{d\zeta_I}}{\| m(\nu,\frac{\partial f}{\partial \zeta})\|^2}
\end{equation}
to $X$, the sum is over all $\nu$-tuples $I = (I_1,\dots,I_\nu)$, where $1 \leq I_1 < \dots < I_\nu \leq N$, and 
where $\widehat{d\zeta_I}$ means that we have removed the factor $d\zeta_{I} := d\zeta_{I_1}\wedge \dots \wedge d\zeta_{I_p}$
from $d\zeta_1 \wedge \dots \wedge d\zeta_N$, and the sign is such that $d\zeta_I \wedge \widehat{d\zeta_I} = d\zeta_1 \wedge \dots \wedge d\zeta_N$ 
(there are also some scalar constants and a fixed frame of a trivial line bundle),
and $m(\nu,\partial f/\partial \zeta)$ denotes the tuple of all $(\nu\times\nu)$-minors of $\partial f/\partial \zeta$.
The Koppelman integral operator $\mathcal{K}$, which is a homotopy operator for the $\dq$-equation on $X$,
is of the form
\begin{equation}\label{eq:AS1}
    (\mathcal{K} \varphi)(z) = \int_{D'} K(\zeta,z) \wedge \varphi(\zeta),
\end{equation}
which takes forms on $D'$ as its input, and outputs forms on $D$.
Here,
\begin{equation} \label{eq:Kdef}
    K(\zeta,z) = \omega_X(\zeta) \wedge \tilde{K}(\zeta,z),
\end{equation}
and $\tilde{K}$ is defined by
\begin{equation*}
    \tilde{K}(\zeta,z) \wedge d\eta_1 \wedge \dots \wedge d\eta_N = h \wedge (g\wedge B)_n,
\end{equation*}
where $(g\wedge B)_n$ denotes the part of $g\wedge B$ of bidegree $(n,*)$, $\eta_i = \zeta_i - z_i$.
The \emph{Hefer form} $h$ is a $(\nu,0)$-form $h = h_1 \wedge \dots \wedge h_\nu$, where $h_i$ is a $(1,0)$-form
satisfying $\delta_\eta h_i = f_i(\zeta) - f_i(z)$ where $\delta_\eta$ is the interior multiplication with
$$2\pi i \sum \eta_j \frac{\partial}{\partial \eta_j} = 2\pi i \sum (\zeta_j - z_j) \frac{\partial}{\partial \eta_j},$$
and we write $h_i = \sum h^j_i d\eta_j$.
The form $g$ is a so-called weight with compact support, defined as follows.
Let $\chi(\zeta)$ be a cut-off function with compact support in $\Omega'$, which is $\equiv 1$ in a neighborhood of $\Omega$,
and let $s(\zeta,z) = \sum s_i(\zeta,z) d\eta_i$ be a $(1,0)$-form such that $\delta_\eta s = 1$, and
which is smooth in $\zeta$ for $\zeta \in \supp \chi'(\zeta)$, and holomorphic in $z \in \Omega$.
Then
\begin{equation*}
    g := \chi - \dq \chi \wedge \big(s+ s(\dq s) + \dots + s(\dq s)^{n-1}\big).
\end{equation*}
If $\Omega$ is the unit ball $B_1(0) \subseteq \C^N$, then (using the general notation $x \bullet y=x_1\cdot y_1+ ... + x_N \cdot y_N$)
one choice of $s$ is
\begin{equation*}
    \sigma = \frac{\overline{\zeta} \bullet d\eta}{2\pi i(\|\zeta\|^2-\bar{\zeta}\bullet z)}.
\end{equation*}
The Bochner-Martinelli form $B$ is defined by
\begin{equation*}
    B := b + b\dq b + \dots + b(\dq b)^{n-1},
\end{equation*}
where
\begin{equation*}
    b := \frac{\partial \|\eta\|^2}{\|\eta\|^2} = \frac{\bar{\eta}\bullet d\eta}{\|\eta\|^2}.
\end{equation*}

We thus get that $\tilde{K}$ is a sum of terms of the forms
\begin{equation*}
    \chi(\zeta) \frac{(\overline{\zeta_i - z_i})}{\|\zeta-z\|^{2n}} h_j(\zeta,z) \widehat{ d\overline{\eta_i}}
\end{equation*}
and
\begin{equation*}
    \dq\chi(\zeta) \frac{\overline{\zeta_i - z_i}}{\|\zeta-z\|^{2\ell}} h_j(\zeta,z)
    s_k(\zeta,z) \widehat{ d\overline{\eta}_i \wedge d\overline{\eta}_k}.
\end{equation*}
Note that since $s_k(\zeta,z)$ is bounded for $z \in \overline{D}$ and $\zeta \in \supp \chi'(\zeta)$,
$\tilde{K}$ is a sum of terms of the form
\begin{equation} \label{eq:Ktildeterm}
    v_j(\zeta,z) \frac{(\overline{\zeta_i - z_i})}{\|\zeta-z\|^{2n}} h_j(\zeta,z) \widehat{ d\overline{\eta_i}},
\end{equation}
where $v_j(\zeta,z) \in L^\infty(D \times D')$.

If $X$ is the affine cone over a smooth projective complete intersection $Y$, this means that we can choose $f$ such that
$f = (f_1,\dots,f_\nu)$, where $f_1,\dots,f_\nu$ are homogeneous polynomials of degree $d_1,\dots,d_\nu$,
and we let $d := d_1 + \dots + d_\nu$, where $d$ is the degree of $Y$.

Since the rows of the $(\nu \times N)$-matrix  $\frac{\partial f}{\partial \zeta}$ are $(d_i-1)$-homogeneous polynomials,
all $(\nu\times\nu)$-minors of $(\partial f)/(\partial \zeta)$ are $(d-\nu)$-homogeneous polynomials in $\zeta$.
The fact that $Y$ is smooth means that $X$ has an isolated singularity at $\{ 0 \}$.
In addition, this means that the common zero-set of the tuple $m(\nu,\frac{\partial f}{\partial \zeta})$ is just the origin.
Since
\begin{equation*}
    \left\|m\left(\nu,\frac{\partial f}{\partial \zeta}(\lambda \zeta)\right)\right\| = \|\lambda\|^{d-\nu}
    \left\|m\left(\nu,\frac{\partial f}{\partial \zeta}(\zeta)\right)\right\|
\end{equation*}
and since $\|m(\nu,\frac{\partial f}{\partial \zeta})\|$ only vanishes at the origin, we get that
\begin{equation*}
    \left\|m\left(\nu,\frac{\partial f}{\partial \zeta}(\zeta)\right)\right\| \sim \| \zeta \|^{d-\nu}.
\end{equation*}
By \eqref{eq:structure-form}, we then get that if we write $\omega = \sum \omega_I \widehat{d\zeta_I}$, then
\begin{equation} \label{eq:str-form-estimate}
    \|\omega_I(\zeta)\| \leq \frac{1}{\|\zeta\|^{d-\nu}}.
\end{equation}
Note also that using $\zeta^k-z^k = (\zeta-z)(\zeta^{k-1} + \zeta^{k-2} z + \dots + z^{k-1})$,
one can chose the Hefer forms $h_i = \sum h^j_i d\eta_j$ such that $h_i^j(\zeta,z)$ are homogeneous polynomials
in $(\zeta,z)$ of degree $d_i-1$. Thus, if we write $h = \sum h_I d\eta_{I}$, then
\begin{equation} \label{eq:hefer}
    |h_I(\zeta,z)| \leq \sum_{\gamma=0}^{d-\nu} \|\zeta\|^{d-\nu-\gamma} \|z\|^\gamma.
\end{equation}
To conclude, using \eqref{eq:Ktildeterm}, \eqref{eq:str-form-estimate} and \eqref{eq:hefer},
the kernel $K(\zeta,z)$ given by \eqref{eq:Kdef} can be expressed as a sum of terms of the form
\begin{equation} \label{eq:Kterm}
    w(\zeta,z) \frac{\|z\|^\gamma}{\|\zeta\|^\gamma} \frac{(\overline{\zeta_j - z_j})}{\|\zeta-z\|^{2n}} \widehat{ d\overline{\eta_J}}
    \wedge \widehat{d\zeta_I},
\end{equation}
where $\gamma \in \{0,\dots,d-\nu\}$ and $w(\zeta,z) \in L^\infty(D\times D')$.

The projection operator $\mathcal{P}$ is defined by
\begin{equation}\label{eq:AS2}
    (\mathcal{P} \varphi)(z) = \int_{D'} P(\zeta,z) \wedge \varphi(\zeta),
\end{equation}
where the integral kernel $P(\zeta,z)$ is defined in a similar way to \eqref{eq:Kdef},
namely,
\begin{equation*}
    P(\zeta,z) = \omega_X(\zeta) \wedge \mathcal{P}(\zeta,z),
\end{equation*}
where
\begin{equation*}
    \tilde{P}(\zeta,z) \wedge d\eta_1 \wedge \dots \wedge d\eta_N = h \wedge g_n,
\end{equation*}
cf. \cite[(5.5)]{AS}. Since $g_n = \dq\chi \wedge s \wedge (\dq s)^{n-1}$,
it has support on $\supp \dq\chi$, where $s$ is smooth in $\zeta$ and holomorphic
in $z$. If we thus assume that $X$ has an isolated singularity inside $D$, then
$\omega(\zeta)$ is smooth on $\supp g_n$, so to conclude,
$P(\zeta,z)$ is smooth in $\zeta$ and $z$, and with compact support in $\zeta$.

\subsection{Mapping properties of the Andersson-Samuelsson Koppelman integral operator}

\smallskip
\begin{proof}[Proof of Theorem~\ref{thm:main1}]
    Due to Lemma~\ref{lma:Lpformsfunctions} and the form \eqref{eq:Kterm} of the integral kernel $K(\zeta,z)$,
    in order to prove that $\mathcal{K}$ give continuous linear maps $L^p_{0,q}(D') \to L^p_{0,q-1}(D)$ and
    $L^\infty_{0,q}(D') \to C_{0,q-1}^\alpha(D)$,
    is enough to prove that integral kernels of the form
    \begin{equation*}
        k_\gamma(\zeta,z) := \frac{(\overline{\zeta_i - z_i})}{\|\zeta-z\|^{2n}} \frac{\|z\|^{\gamma}}{\|\zeta\|^{\gamma}}
    \end{equation*}
    give continuous linear maps $L^p(D') \to L^p(D)$ and $L^\infty(D') \to C^\alpha(D)$,
    where $0 \leq \gamma \leq d-\nu$ is an integer.
    This is Theorem~\ref{thm:lp-estimate} and Theorem~\ref{thm:C0-estimate}, which also give compactness when $p = \infty$.
    It just remains to prove compactness of $\mathcal{K}$ as a continuous linear map $L^p(D') \to L^p(D)$ when $p < \infty$.
    If an integral operator is defined by a bounded integral kernel, it maps $L^p(D') \to L^p(D)$ compactly, see for example
    \cite[Appendix~B]{Ra}. By Lemma~\ref{lem:cpt-cutoff-estimate}, $\mathcal{K}$ can thus be approximated by compact operators,
    and thus, $\mathcal{K}$ is also compact.

    Finally, since $\mathcal{P}$ is defined by a smooth integral kernel with compact support in $\zeta$,
    it maps $L^1(D')$ to $C^{0,1}(\overline{D})$, since
    \begin{equation*}
        |\mathcal{P} \varphi(z)| \leq \|P(\zeta,z)\|_{L^\infty(D'\times D)} \|\varphi\|_{L^1(D')}
    \end{equation*}
    and
    \begin{equation*}
        |\mathcal{P} \varphi(z) - \mathcal{P} \varphi(w)| \leq 
        \|z-w\|\|\frac{\partial P}{\partial \eta}(\zeta,\eta)\|_{L^\infty(D' \times D)} \|\varphi\|_{L^1(D')},
    \end{equation*}
    and it is compact by the Arzel\`a-Ascoli theorem.
\end{proof}

\smallskip
\begin{proof}[Proof of Theorem~\ref{thm:main3}]
    We let $\varphi_k := \mu_k \varphi$ where $\{\mu_k\}_k$ is the cut-off sequence from Section \ref{ssec:cut-off}.
    As in the proof of Theorem~1.3 in \cite{LR}, $\varphi_k$ can be approximated in $L^p(D')$ by smooth forms
    with support away from the origin,
    and using the Koppelman formula of Andersson-Samuelsson, which in particular holds
    for smooth forms, on this approximating sequence of smooth forms, and taking a limit, we get that
    \begin{equation*}
        \varphi_k = \dq \mathcal{K} \varphi_k + \mathcal{K} \dq \varphi_k
    \end{equation*}
    if $q \geq 1$, or
    \begin{equation*}
        \varphi_k = \mathcal{P} \varphi_k + \mathcal{K} \dq \varphi_k
    \end{equation*}
    if $q = 1$.
    
    Note that $\varphi_k \to \varphi$ in $L^p_{0,q}(D')$, and $\mathcal{K}$ maps continuously $L^p_{0,q}(D') \to L^p_{0,q-1}(D)$
    by use of Theorem \ref{thm:main1} (as $\frac{2n}{2n-(d-\nu)} < \frac{2n}{2n-(d-\nu+1)} \leq p)$. So,  
    $\varphi_k \to \varphi$, $\dq \mathcal{K} \varphi_k \to \dq \mathcal{K} \varphi$ (if $q \geq 1$),
    and $\mathcal{P} \varphi_k \to \mathcal{P} \varphi$ (if $q = 0$) in the sense  of distributions on $D$. 
    Thus, it remains to show that $\mathcal{K} \dq \varphi_k \to \mathcal{K} \dq \varphi$
    in the sense of distributions. 
    We split this into two parts by using $\dq \varphi_k = \mu_k \dq \varphi + \dq \mu_k \wedge \varphi$.
    First, we have that $\mu_k \dq \varphi \to \dq\varphi$ in $L^p_{0,q}(D')$,
    and so $\mathcal{K} (\mu_k \dq \varphi) \to \mathcal{K} \dq \varphi$ in the sense of distributions by the argument above.
    It only remains to show that $\mathcal{K} ( \dq \mu_k \wedge\varphi) \to 0$ in the sense of distributions.
    
    To show this, it is convenient to consider the sequence of integral operators
    $$\mathcal{K}_k \varphi := \mathcal{K} ( \dq \mu_k \wedge\varphi)$$
    with integral kernels consisting of parts $\dq\mu_k(\zeta) \wedge k_\gamma(\zeta,z)$ (see the proof of Theorem \ref{thm:main1}).
    
    Using \eqref{eq:cutoff2} and arguing as in the proof of Theorem \ref{thm:main1},
    we see that it is enough to consider a sequence of kernels
    \begin{eqnarray*}
    t_k (\zeta,z) &=& \frac{\chi_k(\|\zeta\|)}{\|\zeta\| \big| \log\|\zeta\|\big|} \cdot \frac{1}{\|\zeta-z\|^{2n-1}} \frac{\|z\|^{\gamma}}{\|\zeta\|^{\gamma}},
    \end{eqnarray*}
    where $\chi_k$ is the characteristic function of $[e^{-e^{k+1}},e^{-e^k}]$ and $0 \leq \gamma\leq d-\nu$ is an integer.
    Thus, Theorem \ref{thm:lp-estimate2} yields $\mathcal{K}_k \varphi = \mathcal{K}(\dq \mu_k\wedge\varphi) \rightarrow 0$
    in $L^p_{0,q}(D)$ if $p < \infty$, and so clearly also in the sense of distributions.
    It is here where we need that $p\geq \frac{2n}{2n-(d-\nu+1)}$.
    In case $p = \infty$, then $\dq \mu_k \wedge \varphi \to 0$ in $L^{p'}_{0,q}(D)$ for any $p' \leq 2n$,
    and thus, as above, $\mathcal{K}_k \varphi \to 0$ in $L^{p'}$ for any $2n \geq p' \geq \frac{2n}{2n-(d-\nu+1)}$,
    and thus also as distributions.
     \end{proof}

\smallskip
\begin{proof}[Proof of Theorem~\ref{thm:main4}]
For $\varphi \in \Dom \dq_s^{(p)}$, let $\{\varphi_j\}_j$ be a sequence as in Lemma \ref{lem:cut-off2}.
We can assume that the $\varphi_j$ are smooth and with support away from the singularity $\{0\}$
(see the proof of Lemma \ref{lem:cut-off2}). Then
\begin{eqnarray*}
\varphi_j &=& \dq \mathcal{K} \varphi_j + \mathcal{K} \dq \varphi_j
\end{eqnarray*}
as in the proof of Theorem \ref{thm:main3}. By the mapping properties of $\mathcal{K}$, Theorem \ref{thm:main1},
we have that $\mathcal{K} \varphi_j \to \mathcal{K}\varphi$ and $\mathcal{K}\dq \varphi_j \to \mathcal{K} \dq \varphi$
in $L^p(D)$. This implies that $\mathcal{K}\varphi\in \Dom \dq_w^{(p)}$ and
\begin{eqnarray*}
\dq \mathcal{K} \varphi &=& \varphi - \mathcal{K} \dq \varphi
\end{eqnarray*}
in the sense of distributions on $X$. As the $\varphi_j$ are bounded, $\{ \mathcal{K} \varphi_j\}_j$ is a sequence of bounded forms
with $\mathcal{K} \varphi_j \to \mathcal{K} \varphi$ and $\dq \mathcal{K} \varphi_j \to \dq \mathcal{K}\varphi$ in $L^p(D)$.
Hence, we obtain $\mathcal{K}\varphi \in \Dom \dq_s^{(p)}$ by Lemma \ref{lem:cut-off2}.
\end{proof}

\bigskip

\bigskip
{\bf Acknowledgments.}
This research was supported by the Deutsche Forschungsgemeinschaft (DFG, German Research Foundation), 
grant RU 1474/2 within DFG's Emmy Noether Programme.
The first author was supported by the Swedish Research Council.
The authors wish to thank the unknown referee for the careful reading and some suggestions
which helped to improve the readability of the paper.

\end{document}